\documentclass[10pt]{amsart}
\usepackage{amsmath}
\usepackage[usenames,dvipsnames]{color}
\usepackage{parskip}
\usepackage{amsfonts}
\usepackage{amscd}
\usepackage[centertags]{amsmath}
\usepackage{amssymb}
\usepackage[all,cmtip]{xy}
\usepackage[english]{babel}
\usepackage{tabularx}
\usepackage{mathtools}
\usepackage{amsxtra}
\usepackage{euscript}
\usepackage[T1]{fontenc}
\usepackage{doc, exscale, fontenc, latexsym, syntonly}
\usepackage{amsfonts}
\usepackage{amsthm}
\usepackage{graphicx}
\usepackage{xcolor}
\usepackage{tikz-cd}
\usepackage{upgreek}

\numberwithin{figure}{section}
\numberwithin{table}{section}

\newcommand{\scat}{\ensuremath{\mathrm{scat}}}

\newcommand{\id}{\ensuremath{\mathrm{id}}}
\newcommand{\TC}{\ensuremath{\mathrm{TC}}}

\newcommand{\pr}{\ensuremath{\mathrm{pr}}}
\newcommand{\SD}{\ensuremath{\mathrm{SD}}}
\newcommand{\sd}{\ensuremath{\mathrm{sd}}}

\newtheorem{theorem}{Theorem}[section]
\newtheorem{definition}{Definition}[section]
\newtheorem{corollary}{Corollary}[section]
\newtheorem{remark}{Remark}[section]
\newtheorem{example}{Example}[section]
\newtheorem{proposition}{Proposition}[section]
\newtheorem{lemma}{Lemma}[section]

\title{Sequential $m$-contiguity distance}

\author{N\.{I}lay Ek\.{I}z Yazici\textsuperscript{1}, Nursultan Kuanyshov, Ayşe Borat\textsuperscript{2}}

\date{\today}



\address{\textsc{Nilay Ekiz Yazıcı}
Bursa Technical University\\
Faculty of Engineering and Natural Sciences\\
Department of Mathematics\\
Bursa, Turkiye}
\email{nilay.ekiz@btu.edu.tr}

\address{\textsc{Nursultan Kuanyshov}
Suleyman Demirel University (SDU)\\
School of Information Technologies and Applied Mathematics (SITAM)\\
Department of Mathematics\\
Kaskelen, Kazakhistan}
\email{nursultan.kuanyshov@sdu.edu.kz} 

\address{\textsc{Ayşe Borat}
Bursa Technical University\\
Faculty of Engineering and Natural Sciences\\
Department of Mathematics\\
Bursa, Turkiye}
\email{ayse.borat@btu.edu.tr}

\begin{document}

\begin{abstract}
In this paper, we introduce the notion of sequential $m$-contiguity distance for finitely many simplicial maps as a higher analogue of contiguity distance. This invariant generalizes both higher contiguity distance and $m$-contiguity distance, and provides a combinatorial counterpart of sequential $m$-homotopic distance. We investigate its fundamental properties, including invariance under strong homotopy type, behaviour under compositions, categorical products, and barycentric subdivision. 

Moreover, we define sequential $m$-discrete topological complexity of simplicial complexes. As applications, we characterise this invariant (along with $m$-simplicial LS category) in terms of sequential $m$-contiguity distance and prove that they are invariants of strong homotopy type. Furthermore, we establish inequalities relating $m$-simplicial LS category and $m$-discrete sequential topological complexity, extending classical results from topological complexity theory to the simplicial and $m$-dimensional setting.
\end{abstract}

\maketitle

\footnotetext[1] {The author was supported by the TÜBİTAK BIDEB 2211-A Domestic Doctoral Scholarship Program.}
\footnotetext[2]{The corresponding author.}

\section{Introduction}

Homotopy invariants arising from sectional-type constructions play a central role in algebraic topology and its applications. Classical notions such as the Lusternik--Schnirelmann category and Farber's topological complexity measure, in different ways, the complexity of spaces and maps through local homotopical data \cite{ADDD,LS,Sch,Fa}. In recent years, several discrete and simplicial analogues of these invariants have been developed in order to study finite topological models and combinatorial structures. Among these developments are the simplicial Lusternik--Schnirelmann category introduced in \cite{FTMVV}, the discrete topological complexity introduced in \cite{FTMVMV1}, and the contiguity distance between simplicial maps studied in \cite{BPV}. 

The notion of contiguity provides a combinatorial counterpart of homotopy theory in the simplicial setting. Two simplicial maps are contiguous when the union of their images on every simplex forms a simplex in the target complex. This relation generates the contiguity class of simplicial maps, which serves as the simplicial analogue of homotopy classes. Using this framework, Borat, Pamuk and Vergili introduced the contiguity distance between two simplicial maps in \cite{BPV}, inspired by the homotopic distance of maps introduced in \cite{MVML}. Later, higher analogues of contiguity distance were investigated in \cite{EYB}, while the notion of $m$-contiguity distance was introduced in \cite{EYKB} as a simplicial counterpart of the $m$-homotopic distance defined in \cite{MVMLO}. 

Motivated by these developments, in this paper we introduce and systematically study the \emph{sequential $m$-contiguity distance} for finitely many simplicial maps. This invariant simultaneously generalizes higher contiguity distance and $m$-contiguity distance, and provides a natural simplicial analogue of higher $m$-homotopic distance. Roughly speaking, given simplicial maps 
$
\varphi_1,\dots,\varphi_n:K\to K',
$
the sequential $m$-contiguity distance measures how far these maps are from being in the same contiguity class after precomposition with simplicial maps from $m$-dimensional simplicial complexes.

We establish several fundamental properties of this invariant. In particular, we prove invariance under strong homotopy type (Theorem \ref{mSD}), investigate its behaviour with respect to compositions and categorical products (Theorem \ref{categoricalproduct}), and study its relation under barycentric subdivision (Theorem \ref{SD-barycentric}). We also show several monotonicity and functoriality properties analogous to those satisfied by classical homotopic distance invariants.

As applications, we introduce the notion of \emph{$m$-discrete sequential topological complexity} and we characterise this new concept along with \emph{$m$-simplicial Lusternik--Schnirelmann category} in terms of sequential $m$-contiguity distance. More precisely, we show that the $m$-simplicial LS category can be expressed as the sequential $m$-contiguity distance of suitable inclusion-type simplicial maps, while the $m$-discrete sequential topological complexity coincides with the sequential $m$-contiguity distance of projection maps. These characterisations allow us to derive several structural properties and prove that both invariants are preserved under strong homotopy equivalence.

Finally, we establish inequalities relating these invariants, extending the classical relationship between Lusternik--Schnirelmann category and sequential topological complexity to the simplicial and $m$-dimensional setting. In particular, for an edge-path connected simplicial complex $K$, we prove
$
\scat_m(K^{r-1}) \leq \TC_r^m(K) \leq \scat_m(K^r),
$
which generalizes the corresponding inequalities for discrete sequential topological complexity.

The paper is organised as follows. In Section 2, we recall the necessary preliminaries concerning contiguity classes, simplicial LS category, and strong homotopy equivalence. In Section 3, we introduce sequential $m$-contiguity distance and establish its basic properties. Section 4 studies categorical products and proves a categorical product inequality. In Section 5, we investigate the behaviour of the invariant under barycentric subdivision. Finally, in Section 6, we present applications to $m$-simplicial LS category and $m$-discrete sequential topological complexity.

\section{Preliminary}
In this section, we briefly state the classical notions such as the contiguity class of simplicial maps, the simplical LS category, the contiguity distance. 

 Throughout this paper, all simplicial complexes are assumed to be \textbf{edge-path connected}, which is any constant simplicial maps are in the same contiguity class. More details can be seen in \cite{FTMVMV1}.  

We recall that given two simplicial maps $\varphi, \psi : K \to L$ are called \textbf{contiguous} if for  every simplex $\{v_0, \dots, v_k\}$ in $K$, $$\{\varphi(v_0), \dots , \varphi(v_k),\psi(v_0), \dots , \psi(v_k)\}$$ constitutes a simplex in $L$. Such maps are denoted by $\varphi \sim_c \psi$. Furthermore, given two simplicial maps $\varphi, \psi : K \rightarrow L $ are said to be in the same \textbf{contiguity class} if one can find a finite sequence of simplicial maps $\varphi_i : K \rightarrow L$ for $i=0,1,\dots m$ such that $$\varphi = \varphi_1 \sim_c \varphi_2 \sim_c \dots \sim_c \varphi_m= \psi.$$ Such maps are denoted by $\varphi \sim \psi $.

\begin{definition}\cite{BPV} For simplicial maps $\varphi, \psi:K\rightarrow K'$, the contiguity distance between $\varphi$ and $\psi$, denoted by $\SD(\varphi,\psi)$, is the least integer $k\geq 0$ such that there exists a covering of $K$ by subcomplexes $K_0,K_1,...,K_k$ with the property that $\varphi|_{K_j}$ and $\psi|_{K_j}$ are in the same contiguity class for all $j=0,1,...,k.$ If there is no such covering, it is defined to be $\SD(\varphi,\psi)=\infty$.
\end{definition}

\begin{definition} \cite{FTMVV}
Let $K$ be a simplicial complex and $\Omega\subset K$ be a subcomplex. If the inclusion $i:\Omega\hookrightarrow K$ and a constant simplicial map $c_{v_0}: \Omega\rightarrow K$, where $v_0 \in K$ is some fixed vertex, are in the same contiguity class, then $\Omega$ is called categorical.
\end{definition}

\begin{definition} \cite{FTMVV} Let $K$ be a simplicial complex. The simplicial Lusternik-Schnirelmann category $\scat(K)$ is the least integer $k\geq 0$ such that one can find categorical subcomplexes $K_0,K_1,\dots, K_k$ of $K$ covering $K$.
\end{definition}

\begin{definition}
 The simplicial complexes $K$ and $K'$ have the same strong homotopy type if there exist simplicial maps $\varphi:K\to K'$ and $\psi:K'\to K$ such that $\varphi\circ \psi\sim id|_{K'}$ and $\psi\circ\varphi\sim id|_{K}$ and it is denoted by $K\sim K'$.
\end{definition} 

\begin{definition}
    Let $v_{0}$ be a vertex in a simplicial complex $K$. K is strong collapsible if $K$ and $v_{0}$ have the same strong homotopy type, i.e $K\sim v_{0}$. 
\end{definition}

\section{Sequential $m$-contiguity distance}

Following the main ideas of \cite{EYB} and \cite{EYKB}, one can naturally define the sequential m-contiguity distance of simplicial maps. 

\begin{definition}
For simplicial maps $\varphi_1,\dots,\varphi_n:K\rightarrow K'$, the sequential m-contiguity distance denoted by $\SD_m(\varphi_1,\dots,\varphi_n)$, is the least integer $k\geq 0$ such that there exists a covering of $K$ by subcomplexes $K_0,\dots,K_k$ with the property that $\varphi_1\circ\eta,\varphi_2\circ\eta,\dots,\varphi_n\circ\eta$ are in the same contiguity class for any simplicial map $\eta:P\rightarrow K_j$ from an $m$-dimesional simplicial complex $P$.
\end{definition}

For the completeness of the work, we state the following easy proposition without proof.  

\begin{proposition} \label{Prop1}
	For simplicial maps $\varphi_1,\dots,\varphi_n:K\rightarrow K'$, we have the following
\begin{itemize}
	    \item[(1)]  $\SD_m(\varphi_1,\dots,\varphi_n)=\SD_m(\varphi_{\sigma(1)},\dots,\varphi_{\sigma(n)})$ holds for any permutation $\sigma$ of $\{1,\dots,n\}$.
	\item[(2)]  If $ \varphi_i \sim\varphi_{i+1}$ for each $i \in \{1,\dots,n-1\}$ then $\SD_m(\varphi_1,\dots,\varphi_n)=0.$
    \item[(3)] $\SD_m(\varphi_1,\dots,\varphi_n)\leq \SD(\varphi_1,\dots,\varphi_n).$
    \item[(4)] $\SD_m(\varphi_1,\dots,\varphi_t)\leq \SD_m(\varphi_1,\dots,\varphi_n)$ provided that $1<t<n$.
    \item[(5)] If $t\leq m$, then $\SD_t(\varphi_1,\dots,\varphi_n)\leq \SD_m(\varphi_1,\dots,\varphi_n).$ 
    \end{itemize}
\end{proposition}

The other direction of Proposition \ref{Prop1}(2) does not hold in general, as the following example shows.  

\begin{example}
    Consider the simplicial complex $K$ given in Figure~\ref{f1}. Let $\varphi_1:K\to K$ be a simplicial map defined  by

\[    \varphi_1(\{0\})=\varphi_1(\{3\})=\varphi_1(\{4\})=\{0\}, \hspace{0.1in}\varphi_1(\{1\})=\{1\}, \hspace{0.1in}\varphi_1(\{2\})=\{2\}
    \]

and consider the constant simplicial maps $\varphi_2, \varphi_3:K\rightarrow K$ at the vertices $\{0\}$ and $\{1\}$, respectively. 

Let $P$ be an arbitrary 1-dimensional simplicial complex and $\eta:P\to K$ be an arbitrary simplicial map.  

For any $i,j=1,2,3$ and for any $\sigma\in P,$ 
$\varphi_i\circ\eta(\sigma)\cup\varphi_j\circ\eta(\sigma)$ is either a 1-simplex or 0-simplex and it must be contained in $K$ since $K$ is a complete graph.



 Thus, $\varphi_1\circ\eta\sim\varphi_2\circ\eta\sim\varphi_3\circ\eta.$ So we conclude that $\SD_1(\varphi_1,\varphi_2,\varphi_3)=0.$ 

On the other hand, following from the Example 2.3 in \cite{EYB}, we know that $\varphi_i\nsim\varphi_j$ for each distinct $i,j \in \{1,2,3\}$.

\newcommand\size{1}
\begin{figure}
  
\begin{tikzpicture}
     \def\size{2} 
    \draw[thick]  (18:\size) \foreach \a [count=\i] in {90,162,234,306}{ -- (\a:\size) } -- cycle;
    
    \draw[thick] (18:\size) \foreach \a [count=\i] in {162,306,90,234} { -- (\a:\size) } -- cycle;
    
    \foreach \i/\a in {0/18, 1/90, 2/162, 3/234, 4/306} {
        \node[black, fill=black, circle, inner sep=2pt, label={[label distance=-3pt] \a:\i}] at (\a:\size) {};
    }
    
\end{tikzpicture} 
\caption{} \label{f1}
\end{figure} 
\end{example}

\begin{proposition} \label{prop32}
	Given simplicial maps $\varphi_i:K\rightarrow K'$ and $\psi_i:K\rightarrow K'$ for $i\in \{1,...,n\}.$ If $\varphi_i\sim\psi_i$ for each i, then $\SD_m(\varphi_1,...,\varphi_n)=\SD_m(\psi_1,...,\psi_n).$
\end{proposition}

\begin{proof}
    Let $\SD_m(\varphi_1,...,\varphi_n)=k.$ Then there exist subcomplexes $K_0,...,K_k$ of $K$ such that each $K_j$ has the property that $\varphi_i\circ\eta\sim\varphi_{i+1}\circ\eta$ for any map $\eta:P\to K_j$ from $m$ dimensional simplicial complex P and for each $i\in \{1,\dots,n-1\}.$ Moreover, by assumption, we have $\varphi_i\circ\eta\sim\psi_i\circ\eta.$ Hence we conclude that $\SD_m(\psi_1,...,\psi_n)\leq k.$
    
\end{proof}


\begin{proposition} \label{3.3}
	Let $\varphi_1,\varphi_2,\ldots,\varphi_n:K\rightarrow K'$ be simplicial maps and $\mu:M\rightarrow K$ be a simplicial map. Then we have 
	$$ \SD_m(\varphi_1\circ\mu,...,\varphi_n\circ\mu)\leq \SD_m(\varphi_1,...,\varphi_n). $$
\end{proposition}
\begin{proof}
    Let $\SD_m(\varphi_1,...,\varphi_n)=k.$ Then there exist subcomplexes $K_0,...,K_k$ of $K$ such that each $K_j$ has the property that $\varphi_1\circ\eta\sim\dots\sim\varphi_n\circ\eta$  for any map $\eta:P\to K_j$ from $m$-dimensional simplicial complex $P$. Define $M_j:=\mu^{-1}(K_j)$. Take any simplicial map $\eta':P\to M_j.$  Then 
    $$\varphi_s\circ\mu\circ\eta'\sim\varphi_t\circ\mu\circ\eta'$$
    for all $s,t\in \{1,...,n\}$, since $\mu\circ \eta'$ is a map from $P$ to $K_j$. Therefore  $\SD_m(\varphi_1\circ\mu,...,\varphi_n\circ\mu)\leq k.$
\end{proof}

\begin{proposition}\label{right}
    Let $\varphi_1,\dots,\varphi_n:K\to K'$ and $\mu_1,\dots,\mu_n:M\to K$ be simplicial maps. If $\mu_1\sim\dots\sim\mu_n$, then we have $$\SD_m(\varphi_1\circ\mu_1,\dots,\varphi_n\circ\mu_n)\leq \SD_m(\varphi_1,\dots,\varphi_n)$$ 
\end{proposition}
\begin{proof}
    Let $\SD_m(\varphi_1,\dots,\varphi_n)=k.$ Then there exist subcomplexes $K_0,\dots,K_k$ of $K$ such that each $K_j$ has the property that for any map $\eta:P\to K_j$ from $m$-dimensional simplicial complex $P$, $\varphi_1\circ\eta\sim\dots\sim\varphi_n\circ\eta.$ Define $M_j:=\mu_1^{-1}(K_j)$ and take any simplicial map $\eta':P\to M_j.$ Then 
$$\varphi_s\circ\mu_1\circ\eta'\sim\varphi_t\circ\mu_1\circ\eta'$$
    for all $s,t\in \{1,...,n\}$, since $\mu_1\circ \eta'$ is a map from $P$ to $K_j$. Using $\mu_1\sim\dots\sim\mu_n$, we conclude that 
    $\SD_m(\varphi_1\circ\mu_1,\dots,\varphi_n\circ\mu_n)\leq k.$
    
\end{proof}

\begin{corollary}
    Let $\mu_{1},\cdots,\mu_{n}:M\to K$ be simplicial maps that satisfy the condition in Proposition \ref{right} which have right strong homotopy inverse. Then for simplicial maps $\varphi_{1},\cdots,\varphi_{n}:K\to K'$,
    $$\SD_m(\varphi_1\circ\mu_1,\dots,\varphi_n\circ\mu_n)= \SD_m(\varphi_1,\dots,\varphi_n)$$
\end{corollary}

\begin{proof}
    Since $\mu_i$ has right homotopy inverse (for all $i=1,\dots,n$), we have $\mu_i\circ\alpha\sim\id_K.$ Moreover, since $\mu_i\sim \mu_{i+1}$ for all $i$, the same $\alpha$ can be used for each $i$. It follows that $\varphi_i\circ\mu_i\circ\alpha\sim\varphi_i.$ Thus,
    \begin{eqnarray*}
	\SD_m(\varphi_1,...,\varphi_n)&=&\SD_m(\varphi_1\circ\mu_1\circ\alpha,...,\varphi_n\circ\mu_n\circ\alpha)\\
	&\leq& \SD_m(\varphi_1\circ\mu_1,...,\varphi_n\circ\mu_n)\\
	&\leq& \SD_m(\varphi_1,...,\varphi_n)
    \end{eqnarray*}
     where the equality follows from Proposition \ref{prop32} and the inequalities follow from Proposition \ref{right}. Hence, we have $\SD_m(\varphi_1\circ \mu_1,...,\varphi_n\circ\mu_n)=\SD_m(\varphi_1,...,\varphi_n).$
\end{proof}

\begin{proposition} \label{left}
     Let $\varphi_1,\dots,\varphi_n:K\to K'$ and $\mu_1,\dots,\mu_n:K'\to M$ be simplicial maps. If $\mu_1\sim\dots\sim\mu_n$, then we have $$\SD_m(\mu_1\circ\varphi_1,\dots,\mu_n\circ\varphi_n)\leq \SD_m(\varphi_1,\dots,\varphi_n)$$ 
\end{proposition}
 \begin{proof}
     Suppose that $\SD_m(\varphi_1,\dots,\varphi_n)=k.$ Then there exist subcomplexes $K_0,\dots,K_k$ of $K$ such that each $K_j$ has the property that for any map $\eta:P\to K_j$ from $m$-dimensional simplicial complex $P$, $\varphi_1\circ\eta\sim\dots\sim\varphi_n\circ\eta.$ Then
     $$\mu_s\circ\varphi_{s'}\circ\eta\sim\mu_s\circ\varphi_{t'}\circ\eta\sim\mu_t\circ\varphi_{t'}\circ\eta$$
     for all $s,s',t,t'\in \{1,\dots,n\}.$ So $\SD_ m(\mu_1\circ\varphi_1,\dots,\mu_n\circ\varphi_n)\leq k.$
 \end{proof}

 \begin{corollary}
    Let $\mu_{1},\cdots,\mu_{n}:K'\to M$ be simplicial maps that satisfy the condition in Proposition \ref{left} which have left strong homotopy inverse. Then for simplicial maps $\varphi_{1},\cdots,\varphi_{n}:K\to K'$,
    $$\SD_ m(\mu_1\circ\varphi_1,\dots,\mu_n\circ\varphi_n)= \SD_m(\varphi_1,\dots,\varphi_n).$$
\end{corollary}

\begin{proof}
    If $\mu_i$ has left strong homotopy inverse (for all $i$), then $\beta\circ\mu_i\sim id_{K'}.$ Since $\mu_i\sim\mu_{i+1}$ for all $i$, the same $\beta$ can be used for each $i$. It follows that $\beta\circ\mu_i\circ\varphi_i\sim\varphi_i.$ Thus
    \begin{eqnarray*}
		\SD_m(\varphi_1,...,\varphi_n)&=& \SD_m(\beta\circ\mu_1\circ\varphi_1,...,\beta\circ\mu_n\circ\varphi_n) \\
		&\leq& \SD_m(\mu_1\circ\varphi_1,...,\mu_n\circ\varphi_n) \\
		&\leq& \SD_m(\varphi_1,...,\varphi_n)
	\end{eqnarray*}
	where the equality follows from Proposition~\ref{prop32} and the inequalities follow from Proposition~\ref{left}. Hence, we have $ \SD_m(\mu_1\circ\varphi_1,...,\mu_n\circ\varphi_n)=\SD_m(\varphi_1,...,\varphi_n). $
\end{proof}







The preceding propositions show that $\SD_{m}$ is invariant under composition with simplicial maps admitting strong homotopy inverses. We now combine these results to obtain the strong homotopy invariance of $\SD_{m}$.

\begin{theorem}\label{mSD}
   Let $\beta_1,\dots,\beta_n:K'\to K$ and $\alpha_1,\dots,\alpha_n:L\to L'$ has right and left strong homotopy inverses, respectively. Assume in addition that $\beta_i\sim\beta_j $ and $\alpha_i\sim\alpha_j$ for all $i,j.$ If the simplicial maps $\varphi_1,\dots,\varphi_n:K\to L$ and $\psi_1,\dots,\psi_n:K'\to L'$ make the following diagram commutative up to contiguity for each $j=1,\dots,n$, then we have $\SD_m(\varphi_1,\dots,\varphi_n)=\SD_m(\psi_1,\dots,\psi_n).$
   \begin{displaymath}
	\xymatrix{
		K \ar[r]^{\varphi_j}  &
		L \ar[d]^{\alpha_j} \\
		K' \ar[r]_{\psi_j} \ar[u]_{\beta_j} & L' }
	\end{displaymath}
\end{theorem}

\begin{theorem} \label{dimen}
    For simplicial maps $\varphi_1,\ldots,\varphi_n:K\to K'$ we have $$\SD_{dim(K)}(\varphi_1,\ldots,\varphi_n)=\SD(\varphi_1,\ldots,\varphi_n).$$
\end{theorem}
\color{black}
\begin{proof}
    Let $m=dim(K).$ $\SD_m(\varphi_1,\ldots,\varphi_n)\leq \SD_m(\varphi_1,\ldots,\varphi_n) $ is open from Proposition \ref{Prop1}.
    
    For the converse, assume that $\SD_m(\varphi_1,\ldots,\varphi_n)=k.$ There exist subcomplexes $K_0,\dots,K_k$ of K such that each $K_j$ has the property that any map $\eta:P\to K_j$ from an $m$-dimensional simplicial complex P, $\varphi_1\circ\eta\sim \varphi_2\circ\eta\sim\cdots\sim \varphi_n\circ\eta$. Now choose $\eta$ to be surjective with $K_j\subseteq \operatorname{Im}(\eta)$ for each $j$. It follows that $ \varphi_1|_{K_j}\sim\varphi_2|_{K_j}\sim\cdots\sim\varphi_n|_{K_j} $ for every $j=0,\ldots,k$. This completes the proof.

\end{proof}

\begin{proposition} \label{propmu}
    Let $K,K'$ and $K''$ be simplicial complexes and let $\mu,\mu':K''\to K$ and $\varphi_1,\dots,\varphi_n:K\to K'$ be simplicial maps. If $\varphi_1\circ\mu'\sim\dots\sim\varphi_n\circ\mu',$ then $\SD_m(\varphi_1\circ\mu,\dots,\varphi_n\circ\mu)\leq \SD_m(\mu,\mu').$
\end{proposition}

\begin{proof}
    Assume that $\SD_m(\mu,\mu')=k.$ Then there exist subcomplexes $K_0,\dots,K_k$ of $K''$ such that each $K_j$ has the property that for any map $\eta:P\to K_j$ from $m-$ dimensional simplicial complex $P$, $\mu\circ\eta\sim\mu'\circ\eta.$ Then we have 
    $$\varphi_1\circ\mu\circ\eta\sim\varphi_1\circ\mu'\circ\eta$$
	$$\varphi_2\circ\mu\circ\eta\sim\varphi_2\circ\mu'\circ\eta$$
	$$\dots$$
	$$\varphi_n\circ\mu\circ\eta\sim\varphi_n\circ\mu'\circ\eta$$
	  Since $\varphi_1\circ\mu'\sim\dots\sim\varphi_n\circ\mu',$ we find that $\varphi_1\circ\mu\circ\eta\sim\dots\sim\varphi_n\circ\mu\circ\eta.$ Hence, $\SD_m(\varphi_1\circ\mu,\dots,\varphi_n\circ\mu)\leq k.$
\end{proof}

\section{Categorical product}

Since the Cartisian product of two simplical complexes may not be a simplicial complex, we used the categorical product instead of the Cartisian product. Let us denote $K\times L$ for the categorical product of the simplicial complexes $K$ and $L$ which is different from the usual standard notion $K \sqcap L$ used in \cite{K}. Hence, we use $K^{r}$ for categorical product $K\sqcap\cdots\sqcap K$. Note that categorical product is functorial by construction \cite{K}. Here we give the definition of the categorical product of two simplicial complexes, then one can build $K^{r}$ inductively.  

\begin{definition}
    The categorical product of the simplicial complexes $K$ and $L$, denoted by $K\times L$, is defined as follows. The vertices of $K\times L$ are pairs $(v,w)$ of vertices with $v\in K$ and $w\in L$, and the simplices of $K\times L$ are the set of vertices $\{(v_{1},w_{1}),\cdots,(v_{s},w_{s})\}$ such that $(v_{0},\cdots,v_{s})$ is a simplex of K and $(w_{0},\cdots,w_{s})$ is a simplex of L.  
\end{definition}


The categorical product of simplicial maps $f:K\rightarrow L$ and $g:K'\rightarrow L'$ is defined by $f\times g: K\times K'\rightarrow L\times L'$, $(f\times g)(\sigma,\tau):=(f(\sigma),g(\tau))$.

We state the following easy observations as lemma below since we use them in the proof of product formula. 

\begin{lemma}\label{easy lemma}
 For given simplicial maps $\varphi_{i}:K\to L$ and $\psi_{i}:K'\to L'$ with $\varphi_{i}\sim \varphi_{j}$ and $\psi_{i}\sim  \psi_{j}$ for $i,j\in \{1,\cdots,r\}$, then $$\varphi_{i}\times\psi_{i}\sim \varphi_{j}\times \psi_{j}:K\times K'\to L\times L'$$ for all  $i,j\in \{1,\cdots,r\}$.
\end{lemma}

\begin{theorem}\label{categoricalproduct}
    For given simplicial maps $\varphi_{i}:K\to L$ and $\psi_{i}:K'\to L'$ with $\varphi_{i}\sim\varphi_{j}$ and $\psi_{i}\sim\psi_{j}$ for all $i,j\in \{1,\cdots,r\}$. Then 
    
    $$ SD_m(\varphi_{1}\times \psi_{1},\cdots, \varphi_{r}\times \psi_{r})+1 \leq (SD_m(\varphi_{1},\cdots,\varphi_{r})+1)(SD_m(\psi_{1},\cdots,\psi_{r})+1).$$
\end{theorem}

\begin{proof}
Suppose $SD_{m}(\varphi_{1},\cdots,\varphi_{r})=k$ and $SD_{m}(\psi_{1},\cdots,\psi_{r})=l$. Then there exists simplicial complexes $K_{0},\cdots,K_{k}$ for which $\eta:P\to K_{s}$ from $m$-dimensional simplicial complex $P$, $\varphi_{i}\circ \eta$ and $\varphi_{j}\circ \eta$ are in the same contiguity class for all $i,j\in \{1,\cdots,r\}$, $s\in\{1,\cdots,k\}$. Similarly, there exists simplicial complexes $L_{0}\cdots,L_{l}$ for which $\overline{\eta}:P\to L_{\overline{s}}$ from $m$-dimensional simplicial complex $P$, $\psi_{i}\circ \overline{\eta}$ and $\psi_{j}\circ \overline{\eta}$ are in the same contiguity class for all $i,j\in \{1,\cdots,r\}$, $\overline{s}\in\{1,\cdots,\ell\}$.

Consider the following collections $\{K_{s}\times L_{\overline{s}}\}$ for $0\leq s\leq k$, $0\leq \overline{s}\leq l$. By construction, $K\times K'$ are covered by above collection. 

Now let $P$ be an $m$-dimensional simplicial complex with simplicial map $h:P\to K_{s}\times L_{\overline{s}}$. By the definition of categorical product, there exist projections, which are simplicial maps $h_{K_{s}}:P\to K_{s}$ and $h_{L_{\overline{s}}}:P\to L_{\overline{s}}$. By $\SD_{m}$-property for $K_{s}$ and $L_{\overline{s}}$, we have the following:

$$\varphi_{i}\circ h_{K_{s}}\sim \varphi_{j}\circ h_{K_{s}}$$ and $$\psi_{i}\circ h_{L_{\overline{s}}}\sim \psi_{j}\circ h_{L_{\overline{s}}}$$ for all $i,j\in \{1,\cdots,r\}$, $s\in\{1,\cdots,k\}$ and  $\overline{s}\in\{1,\cdots,\ell\}$.  On the other hand, the categorical product satisfies the composition rule, then we have the following:

$$(\varphi_{i}\times \psi_{i})\circ h=(\varphi_{i}\circ h_{K_{s}})\times (\psi_{i}\circ h_{L_{\overline{s}}})$$ for all $i\in \{1,\cdots,r\}$. 

By Lemma \ref{easy lemma}, we conclude $(\varphi_{i}\times \psi_{i})\circ h$ and $(\varphi_{j}\times \psi_{j})\circ h$ are in the same contiguity class for all $i,j\in \{1,\cdots,r\}$. Therefore, the collections  
$\{K_{s}\times L_{\overline{s}}\}$ for $0\leq s\leq k$, $0\leq \overline{s}\leq \ell$ have $\SD_{m}$-property. This completes the proof of the theorem.  
\end{proof}

\section{Barycentric subdivision}
    We recall the barycentric subdivision of simplicial complex K, denoted by $sd(K)$, and prove the monotonicity property for sequential $m$-contiguity distance in this section. 
    
\begin{definition}
The barycentric subdivison of a given simplicial complex $K$ is the simplicial complex $\sd(K)$ whose set of vertices is $K$ and each n-simplex in $\sd(K)$ is of the form $\{\sigma_{0},\sigma_{1},\cdots,\sigma_{n}\}$ where $\sigma_{0}\subsetneq \sigma_{1}\subsetneq,\cdots,\subsetneq\sigma_{n}$. 
\end{definition}

 \begin{definition}
     For a simplicial map $\varphi:K\to L$, the induced map $\sd(\varphi):\sd(K)\to\sd(L)$ is given by $\sd(\varphi)(\{\sigma_{0},\sigma_{1},\cdots,\sigma_{n}\})=\{\varphi(\sigma_{0}),\varphi(\sigma_{1}),\cdots,\varphi(\sigma_{n})\}$
 \end{definition}
Notice that $\sd(\phi)$ is a simplicial map, $\sd(\id)=\id$ and $\sd(\varphi\circ\psi)=\sd(\varphi)\circ\sd(\psi)$. In another word, barycentric subdivision is functorial. 

\begin{proposition}\cite[Proposition 3.1.3]{FTMVMV2}\label{bary}
If the simplicial maps $\phi,\psi:K\to L$ are in the same contiguity class, so are $\sd(\phi)$ and $\sd(\psi).$
\end{proposition}

For the completeness of our paper, we give basic properties and propositions with proofs. We will use them in the proof of our main theorem in this section.

\begin{proposition}\label{barycentric}
If the simplicial maps $\varphi_{i}:K\to L$ are in the same contiguity class for all $\varphi_{i}\sim\varphi_{j}$,  $i,j\in\{1,\cdots,n\}$, so are $\sd(\varphi_{i})$ and $\sd(\varphi_{j}).$
\end{proposition} 

\begin{proof}

We argue by induction on $n$. The case $n=2$ is precisely Proposition \ref{bary}. Assume the statement holds for every family of $n-1$ simplicial maps belonging to the same contiguity class, and let $\varphi_1,\ldots,\varphi_n:K\to L$ belong to the same contiguity class. Then the subfamily $\varphi_1,\ldots,\varphi_{n-1}$ also belongs to the same contiguity class. By the induction hypothesis, $\sd(\varphi_1),\ldots,\sd(\varphi_{n-1})$ belong to the same contiguity class.
Since $\varphi_n$ belongs to the same contiguity class as the other maps, there exists some
$r\in\{1,\ldots,n-1\}$ such that $\varphi_r$ and $\varphi_n$ are in the same contiguity class. Applying Proposition \ref{bary} yields $\sd(\varphi_r)\sim\sd(\varphi_n).$
Because $\sd(\varphi_r)$ already belongs to the contiguity class of
$\sd(\varphi_1),\ldots,\sd(\varphi_{n-1}),$ it follows by transitivity that $\sd(\varphi_1)\sim\cdots\sim\sd(\varphi_n).$ Therefore, $\sd(\varphi_1),\ldots,\sd(\varphi_n)$ belong to the same contiguity class. This completes the proof. 
\end{proof}

The following corollary follows directly from Proposition~\ref{barycentric}.

\begin{corollary}\label{m-barycentric}
Let $\varphi_{i}:K\to L$ be the simplicial maps for $i\in\{1,\cdots,n\}$. If given the simplicial map $\eta:P\to K$ with $\varphi_{i}\circ\eta$ and $\varphi_{j}\circ\eta$ are in the same contiguity class for all $i,j\in\{1,\cdots,n\}$, then $\sd(\varphi_{i})\circ\sd(\eta)$ and $\sd(\varphi_{j})\circ\sd(\eta)$ are in the same contiguity class for $i,j\in\{1,\cdots,n\}$.
\end{corollary}


The following lemma is needed in the proof of Theorem~\ref{SD-barycentric}.

\begin{lemma}\cite{EYKB}\label{simp-approx}
Let $f:K\to L$ be a simplicial map. Then $f\circ c_K:\sd(K)\longrightarrow L$ is contiguous to $c_L\circ\sd(f):\sd(K)\longrightarrow L,$ where
$c_K:\sd(K)\to K,$ $c_L:\sd(L)\to L$ are the canonical subdivision maps.
\end{lemma}




The following theorem establishes the relationship between the contiguity distance of a family of simplicial maps and the contiguity distance of the maps induced on their barycentric subdivisions.

\begin{theorem}\label{SD-barycentric}
    For simplicial maps $\varphi_{i}:K\to L$ for $i\in\{1,\ldots,n\}$, we have $$\SD_{m}(\sd(\varphi_{1}),\ldots,\sd(\varphi_{n}))\leq \SD_{m}(\varphi_{1},\ldots,\varphi_{n})$$
\end{theorem} 

\begin{proof}
Let $\SD_m(\varphi_1,\ldots,\varphi_n)=k .$ Then there exist subcomplexes $K_0,\ldots,K_k$ of $K$ such that
$K=\bigcup_{\ell=0}^{k}K_\ell$ and, for each $\ell\in\{0,\ldots,k\}$, every $m$-dimensional
simplicial complex $P$, and every simplicial map $\nu:P\longrightarrow K_\ell ,$
the maps $\varphi_1\circ\nu,\ldots,\varphi_n\circ\nu$ belong to the same contiguity class.

Consider the cover $\sd(K_0),\ldots,\sd(K_k)$ of $\sd(K)$. We claim that this cover satisfies the defining property of $\SD_m(\sd(\varphi_1),\ldots,\sd(\varphi_n)).$ 

Let us fix $K_l$. Let $P$ be any $m$-dimensional simplicial complex, and $\eta:P\longrightarrow\sd(K_\ell)$ be any simplicial map.

Since $c\circ\eta:P\to K_\ell$ is a simplicial map where $c:\sd(K_\ell)\to K_\ell$ is the canonical barycentric subdivision map, 
$\varphi_1\circ c\circ\eta,\ldots,\varphi_n\circ c\circ\eta$ belong to the same contiguity class.

By Corollary~\ref{m-barycentric},
$\sd(\varphi_1)\circ \sd(c)\circ\sd(\eta),\ldots,\sd(\varphi_n)\circ \sd(c)\circ\sd(\eta)$  are in the same contiguity class.

On the other hand, since $\sd(c)\circ \sd(\eta)=c_{\sd K_l}\circ \sd(\eta)$, they are also in the same contiguity class. Thus, $\sd(\varphi_1)\circ c_{\sd K_l}\circ \sd(\eta)\sim\ldots\sim\sd(\varphi_n)\circ c_{\sd K_l}\circ \sd(\eta)$. 


If we apply Lemma~\ref{simp-approx} to the simplicial map $\eta:P\to\sd(K_\ell)$, then it  follows that
$\eta\circ c_P\sim_c c_{\sd(K_\ell)}\circ\sd(\eta).$ Hence 
$\sd(\varphi_1)\circ \eta\circ c_P\sim\ldots\sim\sd(\varphi_n)\circ \eta\circ c_P$. 


Finally, since $c_P:\sd(P)\to P$ is a simplicial approximation to the identity,
the simplicial approximation theorem implies that
$\sd(\varphi_i)\circ\eta$ and $\sd(\varphi_i)\circ\eta\circ c_P$ are contiguous.
Hence $\sd(\varphi_1)\circ\eta,\ldots,\sd(\varphi_n)\circ\eta$ belong to the same contiguity class.

Since $\sd(K_\ell)$ is arbitary, we obtain that
$\SD_m(\sd(\varphi_1),\ldots,\sd(\varphi_n))\le k.$
\end{proof}





\color{black}
\section{Applications of sequential $m$-contiguity distance}

\subsection{$m$-simplicial LS category}

\begin{definition}\cite{EYKB} Let $K$ be a simplical complex and fix some integer $m\geq 1$. The $m$-simplicial LS category $scat_m(K)$ is the least integer $k\geq 0$ such that $K$ has a cover $K_0,\dots,K_k$ where $K_0,\dots,K_k$ are subcomplexes of $K$ and each $K_j$ has the property that any simplicial map $\eta:P\rightarrow K_j$ from an $m$-dimensional simplicial complex $P$, $\iota_j\circ\eta$ and $c_{v_0}\circ\eta$ are in the same contiguity class where $\iota_j:K_j\rightarrow K$ is the inclusion and $c_{v_0}:K_j\rightarrow K$ is the constant map.

\end{definition}

\begin{theorem}\label{mLS} Let K be a simplicial complex with a fixed vertex $v_{0}$, and let $r<s$. For each $r$-tuple $1\leq i_{1}<\cdots<i_{r}\leq s$, define $\mathcal{I}_{(i_{1},\cdots,i_{r})}:K^{r}\to K^{s}$ by  $\mathcal{I}_{(i_{1},\cdots,i_{r})}(\sigma_{1},\cdots,\sigma_{r})=(\varphi_{1},\cdots,\varphi_{s})$, where 

$$
\varphi_{k}=\begin{cases}
\sigma_{m} \,\,\, \text{if}  \,\, k=i_{m} \\
v_{0} \,\,\, \text{otherwise}
\end{cases}
$$
For simplicity, the elements of the set $\{\mathcal{I}_{(i_{1},\cdots,i_{r})}|1\leq i_{1}<\cdots<i_{r}\leq s\}$ is denoted by $\mathcal{I}_{j}$.  
Then
  $$\SD_m(\mathcal{I}_1,\mathcal{I}_2,\cdots,\mathcal{I}_{\binom{s}{r}})=\scat_m(K^{r})$$  
\end{theorem}

\begin{proof}
Suppose that $\scat_{m}(K^{r})=k$. Then $K^{r}$ has a cover $K_0,\dots,K_k$ where $K_0,\dots,K_k$ are subcomplexes of $K^{r}$ and each $K_j$ has the property that any simplicial map $\eta:P\rightarrow K_j$ from an $m$-dimensional simplicial complex P, $\bar{\iota_j}\circ\eta$ and $\bar{c_{j}}\circ\eta$ are in the same contiguity class where $\bar{\iota_j}:K_j\rightarrow K^{r}$ is the inclusion and $\bar{c_{j}}:K_j\rightarrow K^{r}$ is the constant map.
Here $\bar{\iota_{j}}=(pr_{1},\cdots,pr_{r})$ where $pr_{i}:K^{r}\to K$ is the projection to the i-th factor and $\bar{c_{j}}=(c_{v_{0}},\cdots,c_{v_{0}})$ is the constant map. Note that K is edge-path connected, then one can choose all the constant maps as the constant map $c_{v_{{0}}}:K^{r}\to K$. Thus, $\pr_{1}\circ \overline{\iota}_j\circ\eta\sim c_{v_{0}}\circ\overline{\iota}_j\circ\eta,  \ldots ,\pr_{r}\circ \overline{\iota}_j\circ\eta\sim c_{v_{0}}\circ\overline{\iota}_j\circ\eta$.


Hence we obtain $$pr_{1}\circ\overline{\iota}_j\circ\eta\sim pr_{2}\circ\overline{\iota}_j\circ\eta\sim\cdots \sim pr_{r}\circ\overline{\iota}_j\circ\eta\sim c_{v_{0}}\circ\overline{\iota}_j\circ\eta.$$ 

Notice that the maps $\mathcal{I}_{t}$ can be obtained from $(pr_{1},\cdots,pr_{r},c_{v_{0}},\cdots,c_{v_{0}})$ by permuting the components for all $t=1,\cdots,\binom{s}{r}$. 

Therefore, we obtain that $$\mathcal{I}_{1}\circ\eta\sim \mathcal{I}_{2}\circ\eta\sim\cdots\sim\mathcal{I}_{\binom{s}{r}}\circ\eta$$
for each $K_{j}$. This concludes $\SD_m(\mathcal{I}_1,\mathcal{I}_2,\cdots,\mathcal{I}_{\binom{s}{r}})\leq k$.

To prove the reverse inequality, we assume $\SD_m(\mathcal{I}_1,\mathcal{I}_2,\cdots,\mathcal{I}_{\binom{s}{r}})=k$. Then there exist subcomplexes $K_{0},\cdots,K_{k}$ of $K^r$ such that $\mathcal{I}_{1}\circ\eta\sim \mathcal{I}_{2}\circ\eta\sim\cdots\sim\mathcal{I}_{\binom{s}{r}}\circ\eta$ for all $j=0,\cdots,k$ and all maps $\eta:P\to K_{j}$ from any $m$-dimensional simplicial complex $P$.

Our goal is to show that $\bar{\iota_{j}}\circ \eta$ and $\overline{c}_j\circ \eta$ are in the same contiguity class for all $j=0,\cdots,k$ where $\bar{\iota_{j}}:K_{j}\to K^{r}$ is the inclusion and $\bar{c_{j}}:K_{j}\to K^{r}$ is a constant map.

Let $pr_{i}:K^{r}\to K$ and $\overline{pr_t}:K^s\to K$ be the projection maps for $i=1,\cdots,r$ and $t=1,\cdots,s.$ Since $\binom{s}{r}\geq r$, we obtain the following for each $K_{j}$

$$\overline{\pr_{1}}\circ\mathcal{I}_{1}\circ\eta\sim \overline{\pr_{1}}\circ\mathcal{I}_{2}\circ\eta\sim \cdots\sim \overline{\pr_{1}}\circ\mathcal{I}_{\binom{s}{r}}\circ\eta$$
$$\overline{\pr_{2}}\circ\mathcal{I}_{1}\circ\eta\sim \overline{\pr_{2}}\circ\mathcal{I}_{2}\circ\eta \sim \cdots\sim \overline{\pr_{2}}\circ\mathcal{I}_{\binom{s}{r}}\circ\eta$$
$$\cdots$$
$$\overline{\pr_{s}}\circ\mathcal{I}_1 \circ\eta\sim \cdots \sim\overline{\pr_{s}}\circ\mathcal{I}_{\binom{s}{r}}\circ\eta$$

Notice that it suffices to write only one of the rows. Then we obtain $(\pr_{1},\pr_{2},\cdots,\pr_{r})\sim (c_{v_{0}},c_{v_{0}},\cdots,c_{v_{0}})$ which implies $\bar{\iota_{j}}\circ\eta\sim \bar{c_{j}}\circ \eta.$ Therefore the proof is completed. 
\end{proof}

\begin{remark}
The authors recovered the $m$-simplicial Lusternik Schnirelmann category in \cite{EYKB} in the view of $m$-contiguity distance, provided that $r=1$ and $s=2$.
\end{remark}

By virtue of the inclusions $\mathcal{I}_{(i_{1},\cdots,i_{r})}$ defined in the theorem above, an example demonstrating that the inequality in Theorem~\ref{SD-barycentric} can be strict is provided below.

\begin{example}\label{examplenew}
    Consider the simplicial complex $K$ as given in Figure \ref{f1}. In Example 3.6 in Ref. \cite{FTMVMV2}, it is showed that $\scat(K)=2$ while $\scat(\sd K)=1$. Since $K$ and $\sd K$ are 1-dimensional, by Corollary 3.4 in  \cite{EYKB}, we have $\scat_1(K)=\scat(K)$ and $\scat_1(\sd K)=\scat(\sd K).$

    For a fixed vertex $v_0\in K$, consider the simplicial inclusion maps $\mathcal{I}_1, \mathcal{I}_2, \mathcal{I}_3:K\to K^3$ given by $\mathcal{I}_1(v)=(v,v_0,v_0), \mathcal{I}_2(v)=(v_0,v,v_0)$ and $\mathcal{I}_3(v)=(v_0,v_0,v).$ Hence, by Theorem \ref{mLS}, we have
    $$ 1=\SD_1(\sd\mathcal{I}_1, \sd\mathcal{I}_2, \sd\mathcal{I}_3)<\SD_1(\mathcal{I}_1, \mathcal{I}_2, \mathcal{I}_3) =2. $$
\end{example}

\begin{corollary}
    The m-simplicial Lusternik-Schnirelmann category of simplicial complex $K$, $scat_{m}(K)$, is an invariant of strong homotopy type.  
\end{corollary}
\begin{proof}
    This follows from Theorem \ref{mLS} and Theorem \ref{mSD}. 
\end{proof}

We need the following lemma to prove one of the main results in the sequential $m$-discrete topological complexity, see Theorem~\ref{maininequality}. 

\begin{lemma} \label{inequlaitymscat}
	For simplicial maps $\varphi_{1},\cdots,\varphi_{r}:K\rightarrow K'$, we have $$
	\SD_m(\varphi_{1},\cdots,\varphi_{r})\leq \scat_m(K). $$
\end{lemma}

\begin{proof}
Let $\scat_m(K)=k$. Then there exist subcomplexes $K_0,\dots,K_k$ of $K$ covering $K$ such that each $K_j$ has the property that any map $\eta:P\rightarrow K_j$ from an $m$-dimensional simplicial complex $P$, $\iota_j\circ\eta\sim c_{v_0}\circ \eta$ holds where   $\iota_j:K_j\rightarrow K$ is the inclusion and $c_{v_0}:K_j\rightarrow K$ is a constant map. Applying each simplicial map to the equivalence, we obtain the following 
	$$ \varphi_{1}\circ\eta\sim \varphi_{1}\circ  c_{v_0} \sim c_{1}$$
    $$ \varphi_{1}\circ\eta\sim \varphi_{1}\circ  c_{v_0} \sim c_{2}$$
    $$\cdots$$
    $$ \varphi_{r}\circ\eta\sim \varphi_{r}\circ  c_{v_0} \sim c_{r}$$
    
where $c_{i}$ are constant maps for all $i=1,\cdots,r$. 

Since $K$ is edge-path connected simplicial complex, for each $K_{j}$ we obtain 

$$\varphi_{1}\circ\eta\sim\varphi_{2}\circ\eta\sim\cdots\varphi_{r}\circ\eta.$$ 

This shows that $\SD_m(\varphi_{1},\cdots,\varphi_{r})\leq k$.

\end{proof}

The following is an alternative proof of Lemma \ref{inequlaitymscat}

\begin{proof}
    If we take $K''=K$, $\mu=id:K\to K$ identity map and $\mu'=c_{v_0}:K\to K$ constant map in Proposition \ref{propmu}, then by Proposition \ref{propmu} and by Theorem 3.4 in \cite{EYKB},  we have 
    $$ \SD_m(\varphi_1,\dots,\varphi_n)=\SD_m(\varphi_1\circ id,\dots,\varphi_n\circ id)\leq \SD_m(id,c_{v_0})=\scat_m(K).
    $$
\end{proof}

\begin{proposition} \label{collapse} \cite{EYKB}
    If $K$ is strongly collapsible then $\scat_m(K)=0.$
\end{proposition}


\begin{proposition}
    If $K$ is strongly collapsible then $\SD_m(\varphi_1,\dots,\varphi_n)=0$ for any $\varphi_1,\dots,\varphi_n:K\to K'$ simplicial maps.
\end{proposition}

\begin{proof}
    By Proposition \ref{collapse}, we have $\scat_m(K)=0$. Applying Lemma \ref{inequlaitymscat} to that result, we conclude that $\SD_m(\varphi_1,\dots,\varphi_n)=0.$
\end{proof}

The converse does not hold in general, as shown in the following example.

\begin{example}
    
Consider the simplicial complex $K$ given in Figure \ref{pic}. Let  $id:K\to K $ be the identity map and, for each $i=0,\ldots,5$, let $c_i:K\to K$ denote the constant map sending every vertex of $K$ to the vertex $i$. Let $P$ be a $0$-dimensional simplicial complex. For a simplicial map $\eta:P\to K$, there are exactly six possibilities, denoted by $\eta_0,\ldots,\eta_5$. For each $j=0,\ldots,5$, the image of $\eta_j$ is the vertex $i$. By checking all possible cases, we obtain $$id\circ\eta_j\sim c_0\circ\eta_j\sim c_1\circ\eta_j\sim\ldots\sim c_5\circ\eta_j,$$ for all $j$. Therefore, $\SD_0(id,c_0,\ldots,c_5)=0.$

On the other hand, it is known from \cite{FTMVV} that $scat(K)=1$, and hence $K$ is not strongly collapsible.

\begin{figure}
\centering
\begin{tikzpicture}[scale=0.5]

\fill[gray!40] (0,0) -- (4,6.8) -- (8,0) -- cycle;

\draw (0,0) -- (3,2.8);
\draw (0,0) -- (4,1);
\draw (4,6.8) -- (3,2.8);
\draw (4,6.8) -- (5,2.8);
\draw (8,0) -- (4,1);
\draw (8,0) -- (5,2.8);
\draw (3,2.8) -- (4,1);
\draw (3,2.8) -- (5,2.8);
\draw (5,2.8) -- (4,1);

\draw (0,0) -- (4,6.8);
\draw (0,0) -- (8,0);
\draw (8,0) -- (4,6.8);

\filldraw (0,0) circle (4pt);
\filldraw (4,6.8) circle (4pt);
\filldraw (8,0) circle (4pt);
\filldraw (3,2.8) circle (4pt);
\filldraw (4,1) circle (4pt);
\filldraw (5,2.8) circle (4pt);

\node[below left]  at (0,0) {$0$};
\node[above]       at (4,6.8) {$1$};
\node[below right] at (8,0) {$2$};

\node[left]        at (3,2.8) {$3$};
\node[below]       at (4,1) {$4$};
\node[right]       at (5,2.8) {$5$};

\end{tikzpicture}
\caption{}
\label{pic}
\end{figure}
\end{example}

\subsection{Sequential $m$-discrete  topological complexity}

\begin{definition}
The sequential $m$-discrete  topological complexity $\TC_{r}^m(K)$ of a simplicial complex $K$ is the least non-negative integer $k$ such that $K^{r}$ can be covered by subcomplexes $\Omega_0, \cdots,\Omega_k$ of $K^{r}$, each of which satisfies that there exists simplicial map $\sigma_j : \Omega_j \to K$ such that for any map $\eta_j:P\rightarrow \Omega_j$  from an $m$-dimensional simplicial complex $P$, $\Delta\circ \sigma_j\circ \eta_j \sim \iota_{j}\circ \eta_j$ holds where $\iota_j : \Omega_j \hookrightarrow K^{r}$ is the inclusion and $\Delta:K\rightarrow K^{r}$ is the diagonal map.
\end{definition}

Such $\Omega\subset K^{r}$ subcomplexes are called \textit{Rudyak subcomplex of dimension $m$} or simply \textit{Rudyak subcomplex} if it is clear from the context. Notice that it should be considered as an analogue of the $r$-Farber subcomplex as in \cite{ABCD, FTMVMV1} which refers to the Farber subcomplexes used to set up the $r$-th (sequential) discrete topological complexity. 

\begin{theorem}\label{charc}
    Let $\Omega \subset K^{r}$ be a subcomplex and $\Updelta:K\to K^{r}$ be diagonal map, then the following conditions are equivalent.
    \begin{itemize}
        \item [(1)] $\Omega$ is Rudyak subcomplex.
        \item [(2)] $pr_{i}|_{\Omega}\circ\eta\sim pr_{j}|_{\Omega}\circ\eta$ for all $i,j\in \{1,\cdots,r\}$.
        \item [(3)] One of the restrictions $pr_{1}|_{\Omega}\circ\eta,pr_{2}|_{\Omega}\circ\eta, \cdots, pr_{r}|_{\Omega}\circ\eta$ is a section of $\Updelta$ (up to contiguity).
    \end{itemize}
\end{theorem}

\begin{proof}
    $(1)\implies (2)$. Let $\Omega\subset K^{r}$ be Rudyak subcomplex. Then there exists simplicial map $\sigma : \Omega \to K$ such that for any map $\eta:P\rightarrow \Omega$  from an $m$-dimensional simplicial complex $P$, $\Delta\circ \sigma\circ \eta \sim \iota\circ \eta$ holds. One sees that $\Delta\circ\sigma\circ\eta$ is $r$-tuples of $\sigma\circ\eta$, that is $(\sigma\circ\eta,\cdots,\sigma\circ\eta):P\to K^{r}$ by $(\sigma(\eta(w)),\cdots,\sigma(\eta(w))$ for all $w\in P$. On the other hand, the inclusion map $\iota:\Omega\to K^{r}$ can be written as $(pr_{1},\cdots,pr_{r})$ restricted on $\Omega$.
    Then we observe that $\sigma(\eta(w))\sim pr_{i}|_{\Omega}\circ \eta$ for all $i=1,\cdots,r$.

$(2)\implies (3):$ Fix $i_{0}\in\{1,\cdots,r\}$ and suppose $pr_{i}|_{\Omega}\circ \eta\sim pr_{i_{0}}|_{\Omega}\circ \eta$ for all $i\in \{1,\cdots,r\}$. Then 

$\iota\circ \eta = (pr_{1}|_{\Omega}\circ \eta,\cdots,pr_{r}|_{\Omega}\circ \eta)\sim (pr_{i_{0}}|_{\Omega}\circ \eta,\cdots,pr_{i_{0}}|_{\Omega}\circ \eta)=\Delta\circ pr_{i_{0}}|_{\Omega}\circ \eta$, which means that one of the $pr_{i_{0}}|_{\Omega}\circ \eta$'s is a section of $\Delta$.

$(3)\implies (1)$ Suppose that $pr_{i}|_{\Omega}\circ \eta$ is a section (up to contiguity) of the diagonal map $\Delta$, for some $i\in \{1,\cdots,r\}$ and choose the simplicial map $\sigma:=pr_{i}|_{\Omega}\circ \eta:P\to K$. Then $\Omega$ is Rudyak subcomplex.
    
\end{proof}

\begin{theorem}\label{TC^{m}}
  $\TC_{r}^m(K)=\SD_m(pr_1,pr_2,\cdots,pr_{r}).$  
\end{theorem}
\begin{proof}
    The proof follows from Theorem \ref{charc} (1)-(2). 
\end{proof}

Theorem~\ref{TC^{m}} provides a useful reformulation of the sequential $m$-discrete topological complexity in terms of the $m$-contiguity distance. Consequently, the general properties of $\SD_m$ established in the previous sections immediately imply analogous properties for $\TC_r^m(K)$. We collect some of these consequences in the following corollaries.

\begin{corollary}
    The sequential $m$-discrete  topological complexity of simplicial complex $K$, $\TC_{r}^m(K)$, is an invariant of strong homotopy type.  
\end{corollary}

\begin{proof}
    This follows from Theorem \ref{TC^{m}} and Theorem \ref{mSD}. 
\end{proof}

\begin{corollary}\label{tcmtc}
    $\TC_{r}^m(K)\leq \TC_r(K).$
\end{corollary}

\begin{proof}
   The proof follows from Theorem \ref{TC^{m}}, Proposition \ref{Prop1} (3) and Theorem 2.2 in \cite{EYB}.
\end{proof}

\begin{corollary}
    $\TC_{r}^{dim(K)}(K)=\TC_r(K).$
\end{corollary}

\begin{proof}
    Let $m=dim(K).$ Using Corollary \ref{tcmtc}, we have $\TC_{r}^m(K)\leq\TC_r(K)$. By Theorem \ref{TC^{m}}, Theorem \ref{dimen} and Theorem 2.2 in \cite{EYB}, we obtain $$ \TC_{r}^m(K)=\SD_m(pr_1,pr_2,\cdots,pr_{r})=\SD(pr_1,pr_2,\cdots,pr_{r})=\TC_r(K). $$
\end{proof}

\begin{corollary}
$\TC_{r}^m(K)\leq \TC_{r+1}^m(K)$.
\end{corollary}

\begin{proof}
    Applying Theorem \ref{TC^{m}} and Proposition \ref{Prop1}(5), we obtain $$\TC_{r}^m(K)=\SD_m(pr_1,pr_2,\cdots,pr_r)\leq\SD_m(pr_1,pr_2,\cdots,\pr_{r+1})=\TC_{r+1}^m(K).$$
\end{proof}

\color{black}

\begin{theorem}\label{maininequality}
    For an edge-path connected simplicial complex K, we have $$\scat_{m}(K^{r-1})\leq TC_{r}^{m}(K)\leq \scat_{m}(K^{r})$$
\end{theorem}

\begin{proof}

By Theorem \ref{mLS}, $\scat_m(K^{r-1})=\SD_{m}\big(\mathcal{I}_1,\mathcal{I}_2,\cdots,\mathcal{I}_{\binom{s}{r-1}}\big)$ where $\mathcal{I}_i:K^{r-1}\rightarrow K^s$ as given in Theorem \ref{mLS}. 

Define $J_{(i_{1},\cdots,i_{r-1})}:K^{r}\to K^{s}$ by  $J_{(i_{1},\cdots,i_{r-1})}(v_{1},\cdots,v_{r})=(\beta_{1},\cdots,\beta{s})$, where 

$$
\beta_{k}=\begin{cases}
\pr_{i_1}, \,\,\, \text{if}  \,\, k=i_{1} \\
\vdots \\
\pr_{i_{r-1}}, \,\,\, \text{if}  \,\, k=i_{r-1} \\
\pr_r, \,\,\, \text{otherwise}
\end{cases}
$$

\noindent Hence
$ \mathcal{I}_{(i_1,\cdots,i_{r-1})}=\big(J_{(i_1,\cdots,i_{r-1})}\circ \alpha \big)(v_{1},\cdots,v_{r-1})
$ where $\alpha: K^{r-1}\rightarrow K^s$, $\alpha(v_1, \cdots, v_{r-1})=(v_1, \cdots, v_{r-1},0)$. This implies 

\[
\begin{aligned}
\SD_{m}\big(\mathcal{I}_1,\mathcal{I}_2,\cdots,\mathcal{I}_{\binom{s}{r-1}}\big)
&=\SD_m\big(J_1 \circ \alpha, J_2\circ \alpha, \cdots,J_{\binom{s}{r-1}}\circ \alpha\big)\\
&\leq \SD_m\big(J_1, J_2, \cdots,J_{\binom{s}{r-1}}\big)
\end{aligned}
\]

\noindent where the first inequality follows from Proposition~\ref{3.3}.

On the other hand, by Theorem~\ref{TC^{m}}, $\SD_{m}(pr_{1},\cdots,pr_{r})=\TC_{r}^{m}(K)$. So it remains to show that $\SD_m\big(J_1, J_2, \cdots,J_{\binom{s}{r-1}}\big)\leq \SD_{m}(pr_{1},\cdots,pr_{r}).$

Let $\SD_{m}(pr_{1},\cdots,pr_r)=k$. Then there exist subcomplexes $K_{0},\cdots,K_{k}$ of $K^r$ covering $K^r$ such that $pr_1\circ\eta, pr_2\circ\eta\ldots,pr_r\circ\eta$ are in same contiguity class where $\eta:P\to K_{j}$ is any simplicial map from an $m$-dimensional simplicial complex $P$ for each $j\in\{0,1,\ldots, k\}$. Our aim is to show that $J_1\circ\eta\sim\cdots\sim J_{\binom{s}{r-1}}\circ\eta$. Since $pr_i\circ\eta\sim\pr_j\circ\eta$ for all $i,j=1,\cdots,r, $ it follows that $J_l\circ\eta\sim J_t\circ\eta$ for all $l,t=1,\cdots,\binom{s}{r-1}.$ This completes the first part of the proof.






The second inequality follows immediately from Lemma \ref{inequlaitymscat} and Theorem \ref{TC^{m}}.

\end{proof}

\begin{remark}
    The proof of Theorem \ref{maininequality} can be given explicitly using their definitions similar to \cite{ABCD}, but it is long and tedious. In this paper, we study the sequential $m$-contiguity distance systematically, so we decide to give this short and elegant proof.  
\end{remark}

\section{Acknowledgement} 

This paper has been submitted in partial fulfillment of the requirements for the PhD degree at Bursa Technical University. The second author was partially supported by the grant No. AP25796111 of the Science Committee of the Ministry of Science and Higher Education of the Republic of Kazakhstan. \\

On behalf of all authors, the corresponding author states that there is no conflict of interest.\\

The manuscript has no associated data.

\end{document}